\documentclass[11pt,reqno]{amsart}

\usepackage{a4wide}
\usepackage{amsmath,amssymb} 
\usepackage{bbm}
\usepackage{color}
\usepackage{tikz}
\usepackage{tkz-graph}
\usepackage{tikz-cd} 
\usepackage{dsfont}
\usepackage{amscd}
\usepackage{graphicx}
\usepackage{mathtools}
\usepackage[colorinlistoftodos,prependcaption,textsize=small]{todonotes}
\usepackage{enumerate}

\usepackage{bm}

\theoremstyle{plain}
\newtheorem{theorem}{Theorem}[section]
\newtheorem{prop}[theorem]{Proposition}
\newtheorem{lemma}[theorem]{Lemma}
\newtheorem{coro}[theorem]{Corollary}

\theoremstyle{definition}

\newtheorem{remark}[theorem]{Remark}

\newcommand{\N}{{\mathbb N}}

\newcommand{\mc}{\mathcal}
\newcommand{\im}{{\mathrm{i}}}

\newcommand{\me}{\mathrm{e}}
\newcommand{\dd}{\,\mathrm{d}}
\newcommand{\ts}{\hspace{0.5pt}}

\usepackage{mathtools}

\usepackage{hyperref}

\begin{document}

\title[Parameter dependence of thermodynamic characteristics]{On a family of singular potentials: Parameter dependence of thermodynamic characteristics}

\author{Philipp Gohlke}
\address{Faculty of Mathematics and Computer Science, Friedrich Schiller University, \newline
\hspace*{\parindent}Ernst-Abbe-Platz 2, 07745 Jena, Germany}
\email{philipp.gohlke@uni-jena.de}

\author{Georgios Lamprinakis}
\address{V\"{a}xj\"{o}, Sweden}
\email{georgios.lamprinakis@gmail.com}

\author{J\"{o}rg Schmeling}
\address{Lund University, Centre for Mathematical Sciences, \newline
\hspace*{\parindent}Box 118, 221 00 Lund, Sweden}
\email{joerg@math.lth.se}


\begin{abstract}
We consider the family of singular potentials $\psi_c = 2 \log(|\sin(\pi(x-c))|)$, $c\in \mathbb{T}$ over the doubling map and we examine the dependence of several thermodynamic and multifractal characteristics on the position of the singularity $c$. This includes the pressure functions $\mathcal P(t \psi_c)$, the Birkhoff spectrum of $\psi_c$, and the $L^q$ spectrum of the associated equilibrium measure $\mu_c$. For every $c \in \mathbb{T}$, it is known that $\mu_c$ is given by the diffraction measure of a generalized Thue{\ts}--Morse sequence, with the classical Thue{\ts}--Morse measure arising for $c = 0$. 
If $t\geqslant 0$, we show that $c \mapsto \mathcal{P}(t\psi_c)$ is continuous in $c$. If $t<0$, we prove that the function $c \mapsto \mathcal{P}(t\psi_c)$ is lower semicontinuous but not continuous. In this case, we show that the continuity points are precisely those values $c$ such that $\mathcal{P}(t\psi_c) = \infty$, which form a residual set of vanishing Hausdorff dimension in $\mathbb{T}$. We obtain similar statements about the parameter (semi-)continuity of the $L^q$ spectrum and the Birkhoff spectrum.
\end{abstract}

\maketitle

\section{Introduction}
Multifractal analysis provides a more refined analysis of the statistical properties of the underlying dynamics. For the case of H\"older continuous potentials over expanding systems this is achieved through the tools of the classic thermodynamic formalism. This formalism also yields the scaling properties and the dimension spectrum of the corresponding (unique) equilibrium measure, relying on the fact that such equilibrium measures exhibit a Gibbs property \cite{Bowen, Falconer, Pesin-book, Ruelle}.
In this setting, the so-called multifractal miracle shows the intimate connection of the Birkhoff spectrum of a H\"{o}lder continuous potential $\psi\colon X \to \mathbb{R}$ and the corresponding pressure function, while also revealing some remarkable regularity properties for the corresponding pressure function. More precisely, consider an expanding map $T$ on $X$, and
\[
b_{\psi}(x) = \lim_{n\to \infty}\frac{1}{n}S_n\psi (x), \quad S_n\psi (x)= \sum_{i=0}^{n-1}\psi(T^ix),
\]
whenever the limit exists,
and the respective dimension spectrum
\[
f_{\psi}(\beta) = \dim_H\{ x\in X \colon b_{\psi}(x)=\beta \}.
\]
The topological pressure $\mathcal{P}_{\text{top}}(\psi)$ as defined in \cite{Bowen}  coincides in this setting with the variational pressure
\[
\mathcal{P}_{\text{var}}(\psi) = \sup  h_{\mu} + \int \psi \dd \mu,
\]
where the supremum is taken over all the $T$-invariant Borel probability measures on the phase space $X$.
For sufficiently regular dynamics like a Bernoulli shift, or the doubling map on the torus, $f_{\psi}(\beta)$ is up to a multiplicative constant given by
\[
p^{\ast}(\beta) = \inf\{ p(t)-t\beta \colon t\in \mathbb{R} \},
\]
the Legendre transform of the pressure function $p(t) := \mathcal{P}_{\text{top}}(t \psi) = \mathcal{P}_{\text{var}}(t \psi)$, which is a strictly convex, real analytic function on $\mathbb{R}$ \cite{PW97, PW01}.
The topological pressure exhibits some continuity properties even for less regular potentials. In particular, for a continuous potential over a compact metric space one has that the function $\psi \mapsto \mathcal{P}_{\text{top}}(\psi)$ is Lipschitz continuous \cite{Pesin-book}.
Generalizations of this multifractal formalism have appeared, deviating in many directions from this classical dynamical setup \cite{Feng-Olivier, Iommi-Todd, Johansson-Jordan-Oberg-Pollicott}.
We refer to  \cite{Climenhaga} for a very comprehensive survey.

In this paper, we are concerned with potentials with weaker regularity conditions. More precisely, we are interested in unbounded potentials. These potentials appear naturally, for example from a number theoretic point of view, in the context of continued fraction expansion \cite{FLWW, LR}.
They make their appearance in other contexts as well, as for example for countable expanding Markov systems, where the corresponding shift spaces are defined over infinite alphabets, arising naturally from the dynamical setting \cite{Iommi-Jordan}; see also \cite{Sarig}. Countable Markov shifts appear quite naturally for the multifractal analysis of unbounded potentials in different settings as well, such as the  Saint Petersburg potential over the doubling map \cite{KLRW};  here coming from the structure of the potential itself.

We consider a family of potentials over the doubling map $(\mathbb{T}, T)$, given by
\[
\psi_c \colon \mathbb{T} \to [-\infty, +\infty), \quad x\mapsto  2 \log (| \sin(\pi (x-c))|),
\]
with $c \in \mathbb{T}$. Each $\psi_c$ is a non-positive function that has a logarithmic pole at $c$. One can verify that the function $g_c(x) = \big(1 -  \cos(2\pi (x-c)) \big)/2 = \sin^2(\pi (x-c))$ is a $g$-function in the sense of Keane \cite{Keane}. Therefore, by the result of Ledrappier \cite{Ledrappier}, $\psi_c = \log g_c(x)$ admits a unique equilibrium, $\mu_c$ with 
\[
\mathcal{P}_{\text{var}}(\psi) = h_{\mu_c}+ \int \psi_c \dd \mu_c = 0.
\]
The equilibrium measure $\mu_c$ can be also represented as an infinite Riesz product
\[
\mu_{c} = \prod_{m=0}^{\infty} \bigl( 1 - \cos(2\pi (2^m x - c)) \bigr),
\]
to be understood as a weak limit of absolutely continuous probability measures on the torus \cite{GKS}. 
For $c=0$, $\psi_0 = 2 \log (| \sin(\pi x)|)$ has the classic Thue{\ts}--Morse measure as its unique equilibrium measure, 
\[
\mu_0=\mu_{\operatorname{TM}} = \prod_{m=0}^{\infty} \bigl( 1 - \cos(2\pi 2^m x) \bigr).
\]
We refer to \cite{Queff} for a survey on results about the Thue{\ts}--Morse measure. More generally, it was shown in \cite[Prop.~2.4]{GKS} that $\mu_c$ is the diffraction measure of the generalized Thue{\ts}--Morse sequence $t^c$, given by
\[
t^c_n = \me^{2\pi \im (c + 0.5) S_2(n)},
\]
where $S_2(n)$ is the sum of digits in the binary expansion of $n$.
It is worth noting at this point that $\mu_{\frac{1}{2}}$ coincides with the point mass at the origin $\delta_0$, while for $c\neq 1/2$, $\mu_c$ is purely singular continuous with full topological support \cite{BCEG}.

These (generalized) Thue{\ts}--Morse sequences also give rise to (generalized) Thue{\ts}--Morse polynomials which have been studied in number theory and  harmonic analysis after Mahler's \cite{Mahler} and Gelfond's \cite{Gelfond} results.
The multifractal properties of $\psi_0$ (and $\mu_{\operatorname{TM}}$) were studied, for example, in \cite{BGKS, GLS}, where the latter reveals, yet again, a very natural, although a priori indirect, connection between unbounded potentials and countable Markov shifts. 

The general case was considered in \cite{FSS18} by Fan, Schmeling and Shen, where they studied the scaling properties of the $L^{\infty}$-norm of the Birkhoff sums for the potential $\psi_c$, $c\in \mathbb{T}$.  A complete multifractal analysis was provided in \cite{FSS22}, by the same authors, that showcases that the connection between the Birkhoff spectrum and the Legendre transform of the pressure function persists, when considering a variant of the (variational) pressure function 
\[
p_c(t) : = \mathcal{P}(t \psi_c) = \sup_{\mu \in \mathcal{M}_c} h(\mu) +   \int t \psi_c(x) \dd \mu(x),
\]
where $\mathcal{M}_c$ denotes the set of all invariant probability measures on $\mathbb{T}$ (under the doubling map $T$) that do not have $c$ in their support. 
The domain of the Legendre transform $p^{\ast}_c$ is an interval $(\alpha(c),\beta(c))$, where the lower endpoint is obtained from ergodic optimization as 
\[
\alpha(c) = \inf_{\mu \in \mc M} \int \psi_c \dd \mu,
\]
with $\mc M$ being the space of $T$-invariant Borel probability measures on $\mathbb{T}$, and similarly the upper endpoint is given by $\beta(c) = \sup_{\mu \in \mc M} \int \psi_c \dd \mu$.

\begin{theorem}[{\cite[Thm.~B]{FSS22}}]
\label{THM:Birkhoff-spectrum}
Let $c\in \mathbb{T} \setminus \{1/2 \}$, then, for all $\beta \in (\alpha(c),\beta(c))$
\[
f_{\psi_c}(\beta) =  \frac{p^{\ast}_c(\beta)}{\log 2}.
\]
Furthermore, $f_{\psi_c}(\beta) =0$ for all $\beta < \alpha(c)$ and $\beta\geqslant \beta(c)$.
\end{theorem}

For $c=0$, a similar result was already established by Baake, Gohlke, Kesseb\"ohmer and Schindler \cite{BGKS}.
Fan, Schmeling and Shen also provided in \cite{FSS22} description of the points $c$ that give an infinite lower bound for the domain of the Legendre transform. More precisely, the interval $(\alpha(c),\beta(c))$ contains a halfline in topologically typical and metrically rare cases.
\begin{theorem}[{\cite[Thm.~A]{FSS22}}]
\label{THM:lower-endpoint}
The set $\mc C_{\infty} := \{c \in \mathbb{T}: \alpha(c) = - \infty \}$ is a residual set of $0$ Hausdorff dimension. Furthermore, the complement of $\mc C_{\infty}$ is dense in $\mathbb{T}$.
\end{theorem}

Theorem~\ref{THM:Birkhoff-spectrum} motivates to examine the regularity properties of the pressure function. 
Observe that $\psi_c(x) = \psi_0(x-c)$, which allows us to interpret the potentials $\psi_c$ as rotated versions of $\psi_0$ (or any other $\psi_{c'}$ as a matter of fact).  With that in mind, we are interested in studying the dependence of the pressure function on the singularity $c\in \mathbb{T}$. 
More precisely, we investigate the continuity properties of the map
\[
\mathbb{T}\to (-\infty, +\infty], \quad c\mapsto p_c(t).
\]
It is worth mentioning here the result of Keller \cite[Theorem~4.2.11]{Keller-book}, where the class of upper semicontinuous potentials $\phi\colon X \to [-\infty, +\infty)$ is considered, showing that under perturbations  of the potential, that uniformly decay to zero, the pressure function is continuous. In our case, this result cannot be applied, as the rotations of the potential give rise to ``perturbations" that cannot be controlled nicely in norm.
In contrast, we can derive from Theorem~\ref{THM:lower-endpoint} that one cannot hope for continuity of the aforementioned map for $t<0$.

\section{Main Results}
The first result concerns the description of parameters $c$ that give infinite pressure, which ultimately causes the discontinuity of the pressure function for $t<0$.

\begin{prop}
\label{PROP:infinite-pressure}
Let $t<0$. Then, $p_c(t) = \infty$ if and only if $c \in \mc C_{\infty}$.
\end{prop}

By Theorem~\ref{THM:lower-endpoint} this means that for every $t<0$, the points $c$ with $p_c(t) = \infty$ and the points $c$ with $p_c(t) < \infty$ are both dense in $\mathbb{T}$. As a result the function $c \mapsto p_c(t)$ is highly discontinuous.
That being said, we prove in Section~\ref{Section:lower}, that the map $c\mapsto p_c(t)$ is lower semicontinuous for all $t\in \mathbb{R}$.
When only non-negative $t$'s are considered one can show even more. In particular, for $t\geqslant 0$ we prove that the pressure function depends continuously on the singularity $c$. We summarize our main result in the following theorem.

\begin{theorem}
\label{THM:main}
For $t\in \mathbb{R}$, let $q_t \colon \mathbb{T}\to (-\infty, +\infty], \; c\mapsto p_c(t)$.
\begin{enumerate}
\item For every $t \in \mathbb{R}$, the map $q_t$ is lower semicontinuous. 
\item For $t<0$ the continuity points of $q_t$ are given by $\mc C_\infty$. In particular, the continuity points form a residual set of $0$ Hausdorff dimension.
\item For $t\geqslant 0$, the map $q_t$ is a continuous function on the torus.
\end{enumerate}
\end{theorem}

This should be compared with the following result of Gohlke, Kesseb\"ohmer and Schindler \cite{GKS} where they prove a stronger statement about the particular case $t=2$. 

\begin{theorem}[{\cite[Thm.~2.7]{GKS}}]
The function $q_2 \colon \mathbb{T}\to (-\infty, +\infty], \; c\mapsto p_c(2)$ is real analytic.
\end{theorem}

We emphasize that this result relies heavily on the interpretation of $q_2(c)$ as the correlation dimension of the measure $\mu_c$ and that the methods are not applicable to other values of $t$.
It remains an open question if the function $q_t$ is more regular than continuous for general $t>0$.

\begin{remark}
For every $c \in \mathbb{T}$, the function $t \mapsto p_c(t)$ is monotonously decreasing and convex. Monotonicity comes simply from the fact that $\psi_c$ is a non-positive function for all $c$. Convexity is a very general property that follows from the definition of the pressure as a supremum over affine functions. Convexity already gives  continuity when we restrict to $\mathbb{R}_+$. 
\end{remark}

In fact, as we  shall see in Section~\ref{Section:Birkhoff}, $t \mapsto p_c(t)$ is Lipschitz continuous on $\mathbb{R}_+$, uniformly in $c$. This will be a stepping stone for proving a result on the parameter continuity of the Birkhoff spectrum in Thoerem~\ref{THM:Birkhoff-spectrum-continuity} below. In conjunction with Theorem~\ref{THM:main} it also gives rise to the following observation.

\begin{coro}
\label{CORO:joint-continuity}
The map $(c,t) \mapsto p_c(t)$ is continuous on $\mathbb{T} \times [0,\infty)$.
\end{coro}

In addition to the pressure itself, it is natural to study how the domain of the Legendre transform $p_c^*$, 
$(\alpha(c),\beta(c))$, depends on $c$.
The behaviour of these endpoints reflects how the singularity at $c$ affects
the range of possible Birkhoff averages.
\begin{theorem}
\label{THM:alpha-beta-behaviour}
The function $c \mapsto \beta(c)$ is continuous on $\mathbb{T}$. On the other hand, $c \mapsto \alpha(c)$ is upper semicontinuous and the continuity points of this map are given by $\mc C_{\infty}$.
\end{theorem}

To this end, we emphasize that the natural variational principle appearing in \cite{FSS22}, by considering the modified variational pressure function, in fact coincides with the classical pressure function. This not only is useful as a proving tool, but it also offers a bridge between the intuitive definition of the pressure function in this regime and the established definition of the pressure function.
The following result is complementing the results in \cite[Thm.~2.3]{GKS}, where the same statement is proven for $t\geqslant0$.
\begin{theorem} \label{THM:variational-pressure-equality}
For all $c \in \mathbb{T}$ and $t< 0$, we have that
\[
\sup_{\mu \in \mc M_c} h(\mu) + \int t \psi_c \dd \mu 
= \sup_{\mu \in \mc M} h(\mu) + \int t \psi_c \dd \mu.
\]
\end{theorem}

The results in Theorem~\ref{THM:main} have some interesting implication for the Birkhoff spectrum and the $L^q$ spectrum. In particular, the regularity properties of the pressure function given in Theorem~\ref{THM:main}, extend to the two spectra through their 
connection with the pressure function established in Theorem~\ref{THM:Birkhoff-spectrum} and in \cite[Thm.~2.6]{GKS} respectively.

\begin{theorem}
\label{THM:Birkhoff-spectrum-continuity}
Let $\beta \in [-2\log(2), \beta(c_0))$ or $\beta > \beta(c_0)$ for some $c_0 \in \mathbb{T} \setminus \{ 1/2\}$. Then, $c \mapsto f_{\psi_c}(\beta)$ is continuous at the point $c_0$.
\end{theorem}

We recall at this point that for a $q \in \mathbb{R}$, the $L^q$-spectrum of a measure $\nu$ is given by
\[
\beta_{\nu}(q) := \lim_{n \to \infty} \frac{1}{n \log 2}
\log \left(
\sum_{I \in I_n} \nu(I)^q
\right)
\]
whenever the limit exists,  where $I_n$ is a partition of $\mathbb{T}$ into half-open intervals of length $2^{-n}$ for every $n \in \N$, restricted to the topological support of $\nu$.
From \cite[Thm.~2.6]{GKS}, for any $c \in \mathbb{T}\setminus \{ 1/2\}$, 
\begin{equation}
\label{EQ:lq-pressure}
\beta_{\mu_c}(q) = \frac{p_c(q)}{\log(2)},
\end{equation}
for all $q\in \mathbb{R}$. For $c = 1/2$, this relation no longer holds. It is then a straightforward fact, that the $L^q$ spectrum inherits the regularity properties of the pressure function for $c\neq \frac{1}{2}$. The exceptional case $c=\frac{1}{2}$ needs to be treated separately.

\begin{prop} \label{PROP:Lq-continuity}
For $t\in \mathbb{R}$, let $\ell_t \colon \mathbb{T} \to (-\infty, +\infty], \; c\mapsto \beta_{\mu_c}(t)$. Then,
\begin{enumerate}
\item For every $t \in \mathbb{R}$, the map $\ell_t$ is lower semicontinuous. 
\item For $t<0$ the continuity points of $\ell_t$ are given by $\mc C_\infty$. In particular, the continuity points form a residual set of $0$ Hausdorff dimension.
\item For $t\geqslant 1/2$, the map $\ell_t$ is a continuous function on  $\mathbb{T}$.
\item For $t \in [0,1/2)$, the map $\ell_t$ is discontinuous at $c=\frac{1}{2}$ and it is a continuous function on $\mathbb{T}\setminus\{1/2\}$.
\end{enumerate}
\end{prop}

This should be compared to the corresponding statement about the family of pressure functions in Theorem~\ref{THM:main}.

\section{Variational pressure}

For our purpose, it is computationally and presentationally convenient to exploit the close relation of the doubling map $T \colon x\mapsto 2x \pmod{1}$ with the full shift space of two letters $\mathbb{X} = \mathcal{A}^{\mathbb{N}}$, with $\mathcal{A}= \{0,1\}$.  Consider the standard metric on $\mathbb{X}$, $d(x,y) = 2^{-k}$ whenever $k$ is the smallest integer with $x_k \neq y_k$ and the shift map $S\colon \mathbb{X} \to \mathbb{X}$ defined by $S(x)_n = x_{n+1}$. Then the system $(\mathbb{X}, S)$ consist of the classical dynamical symbolic representation of the  system $(\mathbb{T}, T)$, through the dyadic expansion of a real number, defining the (factor) map $\pi_2 \colon (x_n)_{n \in \mathbb{N}} \mapsto \sum_{n=1}^{\infty} x_n 2^{-n}$, which topologically semi-conjugates the action of $S$ and $T$. The map $\pi_2$ is $1$-to-$1$ everywhere, except at the dyadic rationals. For every finite word $w \in \{0,1\}^n$ and $n \in \mathbb{N}$, we adopt the following cylinder notation  
$
[w] = \{x \in \mathbb{X} : x_{1} \cdots x_n = w_1 \cdots w_n \}.
$
In what follows, without loss of generality, we can identify $\mathbb{T}$ with the full shift on two symbols $\mathbb{X} = \mathcal{A}^{\mathbb{N}}$, with $\mathcal{A}= \{0,1\}$.

In order to cut out small neighbourhoods around the singularity position $c$, we first identify those cylinders of length $n$ that are close to $c$. More precisely, let 
\[
\mathcal{F}_n = \mathcal{F}_n(c) = \{ w \in \mathcal{A}^n : [w] \cap B_{2^{-(n+1)}}(c) \neq \emptyset \}
\]
where $B_r(b)$ denotes the set of all points with Euclidean distance to $b$ less than $r$.
Note that $\mathcal{F}_n$ consists of either one or two elements. 

Let $\mathbb{X}_n$ be the subshift of finite type (SFT) with forbidden word set $\mathcal{F}_n$. Given a H\"{o}lder continuous function $\psi$ on $\mathbb{X}$, let
\[
\mc P_n(\psi)= \sup_{\mu \in \mathcal{M}(\mathbb{X}_n)} h(\mu) + \int_{\mathbb{X}_n} t \psi_c \dd \mu(x),
\]
where $\mathcal{M}(\mathbb{X}_n)$ denotes the set of shift-invariant probability measures on $\mathbb{X}_n$.

The following can probably be considered folklore. In this formulation, it is a slight adaptation of \cite[Prop.~3.8]{FergusonPollicott}.

\begin{prop}
\label{PROP:SFT-approximation}
For every H\"{o}lder continuous function $\psi$ on $\mathbb{X}$, it is $\lim_{n \to \infty} \mc P_n(\psi) = \mc P(\psi)$.
\end{prop}

\begin{proof}[Sketch of proof.]
The result in \cite[Prop.~3.8]{FergusonPollicott} covers the case where $\mc F_n$ is chosen to be a prefix of $c$ of length $n$. The proof relies on the fact that $\mathbb{X}_n$ is topologically mixing for large enough $n$, which is still true in our setting, see \cite[Prop.~6.3]{GKS}. A key step in the proof of \cite[Prop.~3.8]{FergusonPollicott} is to find lengths $k$ and $n$ such that $k$ is an arbitrarily small fraction of $n$ and that there is a word $x$ of length $k$ that does not appear in the unique word in $\mc F_n$. In our case, adjusting $k$ slightly, it is clearly possible to ensure that $x$ does not appear in any word of $\mc F_n$. The remainder of the proof carries over verbatim.
\end{proof}

\begin{coro}
\label{CORO:Hoelder-pressure-variation}
Let $\psi$ be a H\"{o}lder continuous function on $\mathbb{T}$ and $c \in \mathbb{T}$. Then, we have that
\[
\mc P_{\operatorname{var}}(\psi) = \sup_{\mu \in \mc M_c} h(\mu) + \int \psi \dd \mu.
\]
\end{coro}

\begin{proof}
Let $P = \mc P_{\operatorname{var}}(\psi)$ and $P_n = \mc P(\psi)$.
For the requirements of this proof we may and will identify the full shift $(\mathbb{X},S)$ with the doubling map without loss of generality. 
By Proposition~\ref{PROP:SFT-approximation} we can choose a value $n\in \N$ such that $P_n$ is arbitrarily close to $P$. 
Since $P_n$ is the pressure of $\psi$ over $(\mathbb{X}_n,S)$, we can find an invariant measure $\mu_n$ supported on $\mathbb{X}_n$ such that $h(\mu_n) + \int \psi \dd \mu_n$ is arbitrarily close to $P_n$. Since $\mathbb{X}_n$ avoids a neighbourhood of $c$, we conclude that for all $\varepsilon > 0$ we can find $\mu_n \in \mc M_c$ with
\[
P \geqslant \sup_{\mu \in \mc M_c} h(\mu) + \int \psi \dd \mu 
\geqslant h(\mu_n) + \int \psi \dd \mu_n \geqslant P - \varepsilon,
\]
which gives the desired relation.
\end{proof}

For $c \in \mathbb{T}$ and $\mu \in \mc M_c$, let us set 
\[
p^{\mu}_c(t) = h(\mu) + \int t \psi_c \dd \mu,
\]
such that $p_c(t) = \sup_{\mu \in \mc M_c} p^{\mu}_c(t)$ for all $t \in \mathbb{R}$. With this notation, we rephrase Theorem~\ref{THM:variational-pressure-equality} as follows.

\begin{prop}
\label{PROP:neg-t-pressure-variation}
For all $c \in \mathbb{T}$ and $t< 0$, we have that
\[
\sup_{\mu \in \mc M_c} p_c^{\mu}(t)
= \sup_{\mu \in \mc M} p_c^{\mu}(t).
\]
\end{prop}

\begin{proof}
Let $P = \mc P_{\operatorname{var}}(t\psi) =  \sup_{\mu \in \mc M} p_{c}^\mu(t)$.
For every $N \in \N$, we consider the function $\psi_c^N \colon x \mapsto \max\{ \psi_c(x),- N \}$, which introduces a ``cut-off" for the value of the singularity at $-N$ and defines a continuous function on $\mathbb{T}$. For every measure $\mu \in \mc M$, we have due to the monotone convergence theorem,
\[
\int \psi_c \dd \mu = \lim_{N \to \infty} \int \psi_c^N \dd \mu.
\]
If $P< \infty$, we can choose for every $\varepsilon > 0$, a measure $\mu \in \mc M$ with $|P - p_c^{\mu}(t)| < \varepsilon$ and a value $N \in \N$ such that $p_c^\mu(t)$ differs by no more than $\varepsilon$ from
\[
h(\mu) + \int t \psi_c^N \dd \mu \leqslant \mc P_{\operatorname{var}}(t \psi_c^N).
\]
Since $t \psi_c^N$ is H\"{o}lder continuous, Corollary~\ref{CORO:Hoelder-pressure-variation} implies that we can find a measure $\nu \in \mc M_c$ such that 
\[
h(\nu) + \int t \psi_c^N \dd \nu \geqslant \mc P_{\operatorname{var}}(t \psi_c^N) - \varepsilon \geqslant p_c^{\mu}(t) - 2 \varepsilon \geqslant P - 3 \varepsilon. 
\]
Since for $t<0$, the family $t \psi_c^N$ is monotonically increasing in $n$, we get
\[
P \leqslant h(\nu) + \int t \psi_c^N \dd \nu  + 3 \varepsilon \leqslant h(\nu) + \int t\psi_c \dd \nu + 3 \varepsilon 
\leqslant \sup_{\nu \in \mc M_c} p_c^{\nu} + 3 \varepsilon,
\]
and since $\varepsilon >0$ was arbitrary, the claim follows. 
If $P = \infty$, we can find for every $M \in \N$ a measure $\mu \in \mc M$ such that $p_c^\mu(t) > M$. Similarly as above, we can find $N \in \N$ and a measure $\nu \in \mc M_c$ such that 
\[
p_c^\nu(t) \geqslant
h(\nu) + \int t \psi_c^N \dd \nu
\geqslant \mc P_{\operatorname{var}}(t \psi_c^N) - 1
\geqslant p_c^\mu(t) - 2 > M - 2.
\]
Again this yields that $\sup_{\nu \in \mc M_c} p_{c}^\nu(t) = \infty = P$, since $M$ was arbitrary.
\end{proof}

The case of positive $t \geqslant 0$ is covered by a complementary result in \cite{GKS}.

\begin{theorem}[{\cite[Thm.~2.3]{GKS}}]
\label{THM:pressure-variations}
For all $c \in \mathbb{T}$ and $t \geqslant 0$, we have
\[
\sup_{\mu \in \mathcal{M}_c} h(\mu) +   \int t \psi_c(x) \dd \mu(x) \ 
= \ 
\sup_{\mu \in \mc M} h(\mu) +   \int t \psi_c(x) \dd \mu(x).
\]
\end{theorem}

Combining this with Proposition~\ref{PROP:neg-t-pressure-variation}, we obtain the following.

\begin{theorem}
\label{THM:pressure-variation-full-t}
For all $c \in \mathbb{T}$ and $t \in \mathbb{R}$, we have
\[
p_c(t) =
\sup_{\mu \in \mc M_c} p_{c}^{\mu}(t)
= \sup_{\mu \in \mc M} p_{c}^{\mu}(t).
\]
\end{theorem}

This shows that in fact the modified version of the pressure function considered in \cite{FSS22} is equal to the standard variational pressure function (and equal to an appropriate version of the topological pressure function by the results in \cite{GKS}).

\section{Parameters with infinite pressure}

As an immediate consequence of the results presented in the last section, we obtain that the supremum in the definition of $\beta(c) = \sup_{\mu \in \mc M} \int \psi_c \dd \mu$ can be restricted to the space $\mc M_c$ of measures that avoid the point $c$ in their support. A similar statement holds for $\alpha(c) = \inf_{\mu \in \mc M} \int \psi_c \dd \mu$. It should be noted that the following result is already covered by \cite[Prop.~2.1]{FSS22} in the case that $c \in \mathbb{T}$ is not a periodic point.

\begin{coro}
\label{CORO:alpha-beta-version}
For all $c \in \mathbb{T}$, it is $\alpha(c) = \inf_{\mu \in \mc M_c} \int \psi_c \dd \mu$ and $\beta(c) = \sup_{\mu \in \mc M_c} \int \psi_c \dd \mu$.
\end{coro}

\begin{proof}
Since $0 \leqslant h(\mu) \leqslant \log(2)$ for every $\mu \in \mc M$, we get by Theorem~\ref{THM:pressure-variation-full-t} that
\[
\beta(c) = \lim_{t \to \infty} \frac{1}{t} \biggl( \sup_{\mu \in \mc M} h(\mu) + \int t \psi_c \dd \mu \biggr)
= \lim_{t \to \infty} \frac{1}{t} \biggl( \sup_{\mu \in \mc M_c} h(\mu) + \int t \psi_c \dd \mu \biggr) = \sup_{\mu \in \mc M_c} \int \psi_c \dd \mu,
\] 
which is the required relation for $\beta(c)$. Similarly,
\[
\alpha(c) = \lim_{t \to - \infty} \frac{1}{t} \biggl( \sup_{\mu \in \mc M} h(\mu) + \int t \psi_c \dd \mu \biggr)
= \lim_{t \to - \infty} \frac{1}{t} \biggl( \sup_{\mu \in \mc M_c} h(\mu) + \int t \psi_c \dd \mu \biggr) = \inf_{\mu \in \mc M_c} \int \psi_c \dd \mu,
\] 
which finishes the proof.
\end{proof}

The alternative expression for $\alpha(c)$ is the essential ingredient to show that $\alpha(c) = - \infty$ characterises those points $c$ with $p_c(t) = \infty$ for $t < 0$.

\begin{proof}[Proof of Proposition~\ref{PROP:infinite-pressure}]
By Corollary~\ref{CORO:alpha-beta-version}, $\alpha(c) = - \infty$ if and only if $\inf_{\mu \in \mc M_c}\int \psi_c \dd \mu  = - \infty$. For $t < 0$, we have
\[
p_c(t) = t \left(\inf_{\mu \in \mc M_c} t^{-1} h(\mu) + \int \psi_c \dd \mu \right),
\]
and since $0 \leqslant h(\mu) \leqslant \log(2)$ for all $\mu \in \mc M_c$, we get that $p_c(t)$ is equal to $\infty$ if and only if $\inf_{\mu \in \mc M_c} \int \psi_c \dd \mu = - \infty$.
\end{proof}

\section{Lower semicontinuity}\label{Section:lower}

In this section, we prove the first two statements of Theorem~\ref{THM:main}. In the previous sections, we showed that $\alpha(c)$ and $\beta(c)$ admit variational characterizations analogous to that of the modified pressure $p_c(t)$, and we also described the parameters that lead to infinite pressure when $t<0$. The key observation behind the lower semicontinuity of $c \mapsto p_c(t)$ is that any measure in $\mc M_c$ is supported away from the singularity. Hence, whenever $c'$ is close to $c$, the rotated potentials $\psi_{c'}$ vary in a controlled way on the support of such a measure. This yields lower semicontinuity of $c \mapsto p_c(t)$, as well as the corresponding one-sided semicontinuity properties of $\alpha(c)$ and $\beta(c)$.

\begin{prop}
\label{PROP:lower-semicontinuity}
For each $t \in \mathbb{R}$ the function $c \mapsto p_c(t)$ is lower semicontinuous. Also, $c \mapsto \alpha(c)$ is upper semicontinuous and $c \mapsto \beta(c)$ is lower semicontinuous.
\end{prop}

\begin{proof}
Let $t \in \mathbb{R}$ and $c \in \mathbb{T}$. First let us assume that $p_c(t) \neq \infty$.
For $\varepsilon > 0$, choose $\mu \in \mc M_c$ such that $|p^{\mu}_c(t) - p_c(t)| < \varepsilon$. Since $\mu$ does not have $c$ in its support, there is a radius $r>0$ such that $\mu$ is supported on $\mathbb{T}_r:= \mathbb{T} \setminus B_r(c)$, where $B_r(c)$ is the closed ball of radius $r$ around $c$.   
Note that for all $d(c,c') < r/2$, the restriction of $\psi_{c'}$ to  $\mathbb{T}_r$ is Lispshitz continuous with some Lipshitz constant $L = L(r)$ that can be chosen independent of $c'$. For such $c'$ we obtain $|\psi_c(x) - \psi_{c'}(x)| \leqslant L \, d(c,c')$ for all $x\in \mathbb{T}_r$. Since $\mu$ is supported on $\mathbb{T}_r$, this yields
\begin{align*}
p_c(t) 
& \leqslant h(\mu) + t \int \psi_c \dd \mu + \varepsilon
\leqslant h(\mu) + t \int \psi_{c'} \dd \mu + |t| L \, d(c,c') + \varepsilon.
\\ &\leqslant p_{c'}(t) + |t| L \, d(c,c') + \varepsilon.
\end{align*}
Since $\varepsilon > 0$ was chosen arbitrarily, we get the desired relation $p_c(t) \leqslant \liminf_{c'\to c} p_{c'}(t)$.

By Corollary~\ref{CORO:alpha-beta-version}, we can find $\mu \in \mc M_c$ such that $|\beta(c) - \int \psi_c \dd \mu| < \varepsilon$. With the same choices for $r$ and $L$ as above, we obtain for $d(c,c') < r/2$ that
\[
\beta(c) \leqslant \int \psi_c \dd \mu + \varepsilon \leqslant \int \psi_c' \dd \mu + L d(c,c') + \varepsilon \leqslant \beta(c') + L d(c,c') + \varepsilon,
\]
yielding that $\liminf_{c' \to c} \beta(c') \geqslant \beta(c)$, and hence that $c \mapsto \beta(c)$ is also lower semicontinuous.
If $\alpha(c) > - \infty$, we can similarly find $\mu \in \mc M_c$ with $|\alpha(c) - \int \psi_c \dd \mu| < \varepsilon$, and $r,L$ such that for $d(c,c') < r/2$, we have
\[
\alpha(c) \geqslant \int \psi_c \dd \mu - \varepsilon 
\geqslant \int \psi_{c'} \dd \mu - L d(c,c') - \varepsilon 
\geqslant \alpha(c') - L d(c,c') - \varepsilon,
\]
which gives that $c \mapsto \alpha(c)$ is upper semicontinuous at such points.

If $p_c(t) = \infty$, let $N \in \N$ be arbitrary and choose $\mu \in \mc M_c$ such that $p^{\mu}_c(t) > N$. With the same choices as above, we obtain 
\[
N < h(\mu) + t \int \psi_c \dd \mu \leqslant p_{c'}(t) + |t| L d(c,c'),
\]
which yields $\liminf_{c' \to c} p_{c'}(t) \geqslant N$. Since $N$ can be chosen arbitrarily large, we obtain $\lim_{c' \to c} p_{c'}(t) = \infty$, as required.

Finally, assume that $\alpha(c) = - \infty$. For $N \in \N$, choose $\mu \in \mc M_c$ such that $\int \psi_c \dd \mu < -N$. With choices as above,
\[
-N > \int \psi_c \dd \mu \geqslant \int \psi_{c'} \dd \mu - L d(c,c')
\geqslant \alpha(c') - L d(c,c'),
\]
which gives $\liminf_{c' \to c} \alpha(c') = - \infty$, and hence finishes the proof that $c \mapsto \alpha(c)$ is upper semicontinuous at all points $c \in \mathbb{T}$.
\end{proof}

The lower semicontinuity from Proposition~\ref{PROP:lower-semicontinuity} is already enough to recover continuity at points where the pressure is infinite. Together with the density of parameters with infinite pressure, this yields a sharp description of the continuity set for $t<0$.
\begin{coro}
Let $t < 0$. Then, $c_0$ is a continuity point for $c \mapsto p_c(t)$ if and only if $c_0 \in \mc C_{\infty}$. Similarly, $c_0$ is a continuity point for $c \mapsto \alpha(c)$ if and only if $c_0 \in \mc C_{\infty}$.
In particular, the continuity points form a residual set of $0$ Hausdorff dimension in both cases.
\end{coro}

\begin{proof}
First, we recall from Proposition~\ref{PROP:infinite-pressure} that $p_{c}(t) = \infty$ if and only if $c \in \mc C_{\infty}$ and this forms a residual set of $0$ Hausdorff dimension.
Note that if $p_{c_0}(t) = \infty$, then upper semicontinuity of $c \mapsto p_c(t)$ in $c_0$ is immediate and lower semicontinuity follows from Proposition~\ref{PROP:lower-semicontinuity}, hence $c_0$ is a continuity point. Conversely, If $p_{c_0}(t) \neq \infty$, denseness of $\mc C_{\infty}$ implies that we can always choose a sequence of points $(c_n)_{n \in \N}$ with $p_{c_n}(t) = \infty$ and $\lim_{n \to \infty} c_n = c_0$. Hence, continuity of $c \mapsto p_c(t)$ in $c_0$ is impossible.
The same kind of reasoning applies to the function $c \mapsto - \alpha(c)$.
\end{proof}

\section{Continuity results on the pressure and the $L^q$ spectrum} \label{Section:Continuity-Pressure-Lq}

In this section, we will ultimately prove the continuity of $(c,t) \mapsto p_c(t)$ on $\mathbb{T} \times [0,\infty)$ and $c \mapsto \beta(c)$. Furthermore, we will derive the claimed continuity properties of the map $c \mapsto \beta_{\mu_c}(t)$.
By Proposition~\ref{PROP:lower-semicontinuity}, it remains to prove the upper semicontinuity of $c \mapsto p_c(t)$ for $t\geqslant 0$. The main point is to start from almost maximizing measures for nearby parameters, pass to a weak limit, and control the singular term by using truncated potentials.

\begin{prop}
\label{PROP:p-continuity}
For every $t \geqslant 0$, the map $c \mapsto p_c(t)$ is continuous.
\end{prop}

\begin{proof}
Since the claim is trivial for $t=0$ let us fix some $t>0$ in the following.
Due to the lower semicontinuity established in Proposition~\ref{PROP:lower-semicontinuity}, it suffices to show that 
\[
\limsup_{c'\to c} p_{c'}(t) 
\leqslant p_c(t).
\]
Let $(c_m)_{m \in \N}$ be a sequence in $\mathbb{T}$ that converges to $c$. Given $\varepsilon$, choose for every $m \in \N$ a measure $\mu_m \in \mc M_{c_m}$ such that
\begin{equation}
\label{EQ:p-p-mu-variation}
\left|p_{c_m}(t) - p^{\mu_m}_{c_m}(t)\right| < \varepsilon.
\end{equation}
Let $\mu$ be a weak accumulation point of the measures $(\mu_m)_{m \in \N}$ along a maximising subsequence of $( p_{c_m}(t) )_{m \in \mathbb{N}}$. 
For $N \in \N$ and $c' \in \mathbb{T}$, let $\psi_{c'}^N(x) = \max \{ \psi_{c'}(x), -N \}$ be the continuous cut-off function at $-N$. Given a fixed $N \in \N$, note that $\psi^N_{c_m}$ converges uniformly to $\psi_c^N$ on $\mathbb{T}$, and hence
\[
\limsup_{m \to \infty} \int \psi_{c_m}^N \dd \mu_m = \int \psi^N_c \dd \mu.
\]
Since $\psi_c^N$ converges monotonically to $\psi_{c}$ as $N \to \infty$, we obtain via the monotone convergence theorem that
\[
\lim_{N \to \infty} \int \psi_c^N \dd \mu = \int \psi_c \dd \mu.
\]
Due to the fact that $\psi_{c_m} \leqslant \psi^N_{c_m}$ holds pointwise, we obtain
\[
\limsup_{m \to \infty} \int \psi_{c_m} \dd \mu_m
\leqslant \limsup_{m \to \infty} \int \psi^N_{c_m} \dd \mu_m 
= \int \psi^N_c \dd \mu
\xrightarrow{N \to \infty} \int \psi_c \dd \mu.
\]
Combining this with the upper semicontinuity of the measure-theoretic entropy on $(\mathbb{T},T)$ and recalling that $t>0$, we get 
\[
\limsup_{m \to \infty} p_{c_m}^{\mu_m}(t) 
= \limsup_{m \to \infty} h(\mu_m) + t\int \psi_{c_m} \dd \mu_m 
\leqslant h(\mu) + t\int \psi_c \dd \mu.
\]
By Theorem~\ref{THM:pressure-variations}, this expression is bounded above by $p_c(t)$ and we finally get that
\[
\limsup_{m \to \infty} p^{\mu_m}_{c_m}(t) \leqslant p_c(t).
\]
Due to \eqref{EQ:p-p-mu-variation}, this yields
\[
\limsup_{m \to \infty} p_{c_m}(t) 
\leqslant p_c(t) + \varepsilon
\]
and since $\varepsilon > 0$ was arbitrary, this proves the assertion.
\end{proof}

Having established continuity of the pressure for each fixed $t \ge 0$, we can now transfer this regularity to the upper endpoint $\beta(c)$ through its variational characterization, which could also be interpreted geometrically as the asymptotic slope of $p_c(t)$ as $t \to \infty$.

\begin{coro}
\label{CORO:beta-continuity}
The function $c \mapsto \beta(c)$ is continuous on $\mathbb{T}$.
\end{coro}

\begin{proof}
Note that for all $c \in \mathbb{T}$, we have
\[
\beta(c) = \inf_{t > 0} \sup_{\mu \in \mc M_c} t^{-1} h(\mu) + \int \psi_c \dd \mu 
= \inf_{t>0} t^{-1} p_c(t).
\]
Since $t \mapsto t^{-1} p_c(t)$ is continuous by Proposition~\ref{PROP:p-continuity}, we obtain that $\beta(c)$ is the infimum over a family of continuous function. As such, it is upper semicontinuous. Since $c \mapsto \beta(c)$ is lower semicontinuous by Proposition~\ref{PROP:lower-semicontinuity}, we get that this map is in fact continuous.
\end{proof}

In order to prove that $(c,t)\mapsto p_c(t)$ is continuous on an appropriate region, it will be useful to establish an equicontinuity property of the family $(p_c)_{c \in \mathbb{T}}$ on $\mathbb{R}_+$. As we shall see in the next section, this will also be useful for the Birkhoff spectrum.

\begin{lemma}
\label{LEM:Lipshitz-control}
For every $c \in \mathbb{T}$, the function $p_c$ is Lipshitz continuous with Lipshitz constant $L = 2\log(2)$ on $\mathbb{R}_+$.
\end{lemma}
\begin{proof}
We have that 
\[
\int \psi_c \dd \lambda = -2 \log 2 
\]
where $\lambda$ denotes the Lebesgue measure. Since $p_c(t) \geqslant p_c^{\lambda}(t)$
for every $t \in \mathbb{R}$, we have that 
\begin{equation}
\label{EQ:p_c_lower-bound}
p_c(t) \geqslant \log(2) (1- 2t).
\end{equation}
For $0 \leqslant x < y$, we have due to the convexity of $p_c$ that
\[
p_c(x) \leqslant p_c(0) + \frac{x}{y} (p_c(y) - p_c(0)),
\]
and hence, via \eqref{EQ:p_c_lower-bound} and recalling $p_c(0) = \log(2)$ (the entropy of the doubling map), we get 
\[
p_c(x) - p_c(y) \leqslant (1-x/y) (p_c(0) - p_c(y)) \leqslant (y-x) 2 \log(2).
\]
Note that the pointwise relation $\psi_c \leqslant 0$ forces $p_c(t)$ to be monotonically decreasing in $t$. Thus, we obtain
\[
- 2 \log(2) \leqslant \frac{p_c(y) - p_c(x)}{y-x} \leqslant 0,
\]
which proves that $p_c$ is Lipschitz  continuous with Lipschitz  constant $2 \log(2)$ on $\mathbb{R}_+$.
\end{proof}

We can get the claimed joint continuity result for the map $(c,t)\mapsto p_c(t)$ as an immediate application of the Lipschitz equicontinuity on $t$ and the continuity results on $c$.
\begin{coro}
The map
$
F\colon \mathbb T \times [0,\infty) \to \mathbb R,\,(c,t)\mapsto p_c(t),
$
is continuous. 
\end{coro} 
\begin{proof}
By Proposition~\ref{PROP:p-continuity}, for every fixed $t\geqslant 0$, the map $c\mapsto p_c(t)$ is continuous.
By Lemma~\ref{LEM:Lipshitz-control}, for every $c\in\mathbb T$, the function $t\mapsto p_c(t)$ is Lipschitz on
$\mathbb{R}_+$ with Lipschitz constant $2\log 2$, uniformly in $c$. Hence, if
$(c_n,t_n)\to (c,t)$, then
\[
|p_{c_n}(t_n)-p_c(t)|
\leqslant 
|p_{c_n}(t_n)-p_{c_n}(t)| + |p_{c_n}(t)-p_c(t)|.
\]
The first term tends to $0$ by the Lipschitz equicontinuity, and the second tends to $0$ by
continuity of $c\mapsto p_c(t)$ at the fixed value $t$. Therefore
$p_{c_n}(t_n)\to p_c(t)$.
\end{proof}

Recall from \eqref{EQ:lq-pressure} that $\beta_{\mu_c}(t) = p_c(t)/\log(2)$ for all $c \in \mathbb{T}\setminus\{1/2\}$ and all $q \in \mathbb{R}$. Furthermore, it was already discussed in \cite{FSS22,GKS} that for $c=1/2$, we have that
\[
p_{1/2}(t) = \max \{(1-2t),0 \} \log(2),
\]
for all $t \in \mathbb{R}$ and $\mu_{1/2} = \delta_0$. The latter identity directly entails that $\beta_{\mu_{1/2}}(t) = 0$ for all $t \in \mathbb{R}$. Given the statement about the parameter dependence of the pressure function in Theorem~\ref{THM:main}, this is all we need to prove a corresponding result about the $L^q$ spectrum.

\begin{proof}[Proof of Prop.~\ref{PROP:Lq-continuity}]
By the observations above we have that $\beta_{\mu_{1/2}}(t) \leqslant p_{1/2}(t)/\log(2)$ for all $t \in \mathbb{R}$, with equality if and only if $t \geqslant 1/2$. Since $\beta_{\mu_c}(t) = p_{c}(t)$ for all other values of $c,t$, we obtain that lower semicontinuity of $\ell_t \colon c \mapsto \beta_{\mu_c}(t)$ is directly inherited from the corresponding statement about $c \mapsto p_c(t)$. 
For $t<0$, we observe that $\beta_{\mu_{1/2}}(t) = 0$ is finite whereas the values $c$ with $\beta_{\mu_c}(t) = \infty$ lie dense, and hence $c=1/2$ is not a continuity point for $\ell_t$. It follows that the continuity points of $\ell_t$ and $q_t$ coincide and are given by $\mc C_{\infty}$. For $t \geqslant 1/2$, we have that $\beta_{\mu_c}(t) = p_c(t)/\log(2)$ for all $c \in \mathbb{T}$, which yields continuity of $\ell_t$. 
Conversely, if $t \in [0,1/2)$, we have $\beta_{\mu_{1/2}}(t) < p_{1/2}(t)$, which breaks the continuity of $\ell_t$ at $c=1/2$.
\end{proof}

\section{The Birkhoff spectrum} \label{Section:Birkhoff}
We start with some basic facts about the Legendre transform of $p_c$. Recall from Theorem~\ref{THM:Birkhoff-spectrum} that $p_c^{\ast} = \log(2) f_{\psi_c}$ on its interior domain $(\alpha(c), \beta(c))$. Furthermore, since integration with respect to Lebesgue measure $\lambda$ yields $\int \psi_c \dd \lambda = - 2 \log(2)$, we immediately get from Birkhoff's ergodic theorem that the Birkhoff average $b_{\psi_c}(x)$ equals $-2 \log(2)$ for a set of full Lebesgue measure, and hence also of full Hausdorff dimension. Hence, for $\beta_0 := -2 \log(2)$, we obtain that
\begin{equation}
\label{EQ:Legendre-maximum}
p_c^{\ast}(\beta_0) = \log(2) f_{\psi_c}(\beta_0) = \log(2).
\end{equation}
Clearly, we also have $0 \leqslant p_c^{\ast}(\beta) \leqslant \log(2)$ for all $\beta \in (\alpha(c),\beta(c))$ due to the fact that $f_{\psi_c}$ can only take values in $[0,1]$. For $\beta \in (\beta_0,\beta(c))$, we can restrict the domain for the infimum in the definition of the Legendre transform.

\begin{lemma}
\label{LEM:Legendre-transform-minimizer}
Let $c \in \mathbb{T}$, and assume that $\beta \geqslant \beta_0 = - 2 \log(2)$ and $\beta(c) - \beta \geqslant \Delta$ for some $\Delta>0$. Then,
\[
p_c^{\ast}(\beta) = \min_{0 \leqslant t \leqslant \log(2)/\Delta } (p_c(t) - \beta t).
\]
\end{lemma}

\begin{proof}
For $\beta = \beta_0$, recall from \eqref{EQ:Legendre-maximum} 
that $p_c^{\ast}(\beta_0) = \log(2) = p_c(0)$, and hence the infimum is 
obtained at $t=0$. In the following, let $\beta> \beta_0$.

For $\alpha, \beta \in (\alpha(c),\beta(c))$, let $\delta = (\alpha - \beta)$ and observe that by definition of the Legendre transform we have that
$
p_c^{\ast}(\alpha) \leqslant p_c(t) - \alpha t = p_c(t) - \beta t - \delta t,
$
and hence
\[
p_c(t) - \beta t \geqslant p_c^{\ast}(\alpha) + \delta t,
\]
for all $t \in \mathbb{R}$. With $\alpha = \beta_0$ and using that $p_c^{\ast}(\beta_0) = \log(2)$, we obtain that $\delta < 0$ by assumption and hence for $t <0$,
\[
p_c(t) - \beta t \geqslant \log(2) + \delta t > \log(2) \geqslant p_c^{\ast}(\beta). 
\]
That is, we can restrict the infimum to $t\geqslant 0$ for the calculation of $p_c^{\ast}(\beta)$. Similarly, for $\alpha = \beta + \Delta \in (\alpha(c),\beta(c))$, and $t > \log(2) / \Delta$, we obtain
\[
p_c(t) - \beta t \geqslant p^{\ast}_c(\alpha) + \Delta t \geqslant \Delta t > \log(2) \geqslant p_c^{\ast}(\beta).
\]
As a consequence,
\[
p_c^{\ast}(\beta) = \inf_{0 \leqslant t \leqslant \log(2)/\Delta } (p_c(t) - \beta t)
= \min_{0 \leqslant t \leqslant \log(2)/\Delta } (p_c(t) - \beta t),
\]
where the last equality follows from the observation that $p_c(t) - \beta t$ is a continuous function in $t$ for $t \geqslant 0$, and therefore obtains its minimum on a compact interval.
\end{proof}

To prove continuity of $c \mapsto p_c^{\ast}(\beta)$ in an appropriate region, we now combine the compactness reduction from Lemma~\ref{LEM:Legendre-transform-minimizer} with the joint continuity of the family $(p_c)_{c \in \mathbb{T}}$ established in Section~\ref{Section:Continuity-Pressure-Lq}.

\begin{prop}
\label{PROP:Birkhoff-continuity}
Let $\beta \geqslant -2 \log(2)$ and $\beta \neq \beta(c_0)$ for some $c_0 \in \mathbb{T} \setminus \{ 1/2\}$. Then, $c \mapsto f_{\psi_c}(\beta)$ is continuous at the point $c_0$.
\end{prop}

\begin{proof}
First, let us assume that $\beta > \beta(c_0)$. Since $c \mapsto \beta(c)$ is continuous by Corollary~\ref{CORO:beta-continuity}, we obtain that $\beta > \beta(c)$ for all $c$ sufficiently close to $c_0$. But this implies that $f_{\psi_{c}}(\beta) = 0$ for all such $c$, and we obtain that $f_{\psi_{c_0}}(\beta) = 0$ is locally constant around $c_0$.

If $- 2 \log(2) \leqslant \beta < \beta(c_0)$, we can choose $r>0$ and $\Delta>0$ such that for all $d(c,c_0)<r$, we have $\beta + \Delta < \beta(c_n)$.
By Lemma~\ref{LEM:Legendre-transform-minimizer}, this implies that
\[
p^{\ast}_{c}(\beta) = \min_{0 \leqslant t \leqslant \log(2)/\Delta} (p_c(t) - \beta t),
\]
on $B_r(c_0)$. By Proposition~\ref{PROP:p-continuity}, $c \mapsto p_c(t)$ is a continuous function for $t \geqslant 0$, and hence, as an infimum over continuous functions, $c \mapsto p^{\ast}_c(\beta)$ is upper semicontinuous on $B_r(c_0)$. 
In order to show lower semicontinuity at $c_0$ pick an arbitrary sequence $(c_n)_{n \in \N}$ in $\mathbb{T}$ that converges to $c_0$. For each $n \in \N$ let $t_n \in [0,\log(2)/\Delta]$ be a value such that $p^{\ast}_{c_n}(\beta) = p_{c_n}(t_n) - \beta t_n$. Let $t_0$ be any accumulation point of the sequence $(t_n)_{n \in \N}$ along a minimizing subsequence of $(p_{c_n}^{\ast}(\beta))_{n \in\N}$.
By Corollary~\ref{CORO:joint-continuity}, for $\varepsilon > 0$, we get for all large enough $n$,
\[
p^{\ast}_{c_n}(\beta) = p_{c_n}(t_n) - \beta t_n \geqslant p_{c_0}(t_0) - \beta t_0 - \varepsilon
\geqslant 
p^{\ast}_{c_0}(\beta) - \varepsilon.
\]
This implies that
$
\liminf_{n \to \infty} p^{\ast}_{c_n}(\beta)
 \geqslant p_{c_0}^{\ast}(\beta) - \varepsilon
$,
which shows the claimed lower semicontinuity since $\varepsilon>0$ was arbitrary.
\end{proof}

\begin{remark}
Note that for $\beta< - 2\log(2)$ we have that $f_{\psi_c}(\beta)\equiv 1$ for $c \in \mathcal C_{\infty}$, whereas the domain of $f_{\psi_c}$ is bounded for all $c \notin \mathcal C_{\infty}$ \cite[Thm.~B]{FSS22}. The statement of Proposition~\ref{PROP:Birkhoff-continuity} therefore cannot be extended to the halfline $(-\infty,-2\log(2))$.
On the other hand, it seems natural to try to extend the result in Proposition~\ref{PROP:Birkhoff-continuity} to $\beta = \beta(c_0)$. In this case, small perturbation of $c_0$ turn $\beta$ into a point close to the boundary, and the boundary is infinitely steep. Hence, one would need a detailed control over the joint continuity properties of the family $\{p_c^{\ast}\}_{c \in \mathbb{T}}$ at the boundary of the domain to decide whether the result in Proposition~\ref{PROP:Birkhoff-continuity} can be extended. This seems currently out of reach and we leave it as an open problem for future research. 
\end{remark}

\section*{Acknowledgements} 

The authors want to thank the Institut Mittag-Leffler for its kind hospitality. PG acknowledges support from the German Research Foundation (DFG) through Project
509427705 and Project Z-72-00539-00-16100116.

\end{document}